\documentclass[11 pt, reqno]{amsart}

\usepackage[latin1]{inputenc}
\usepackage{color, amsmath,amsthm,amssymb,amscd}
\usepackage{bbm}

\newtheorem{thm}{Theorem}[section]

 \newtheorem{prop}[thm]{Proposition}

 \theoremstyle{definition}
 
 \theoremstyle{remark}
 \newtheorem{rem}[thm]{Remark}
 
 \numberwithin{equation}{section}

\def\be#1 {\begin{equation} \label{#1}}
\newcommand{\ee}{\end{equation}}

\def\sqw{\hbox{\rlap{\leavevmode\raise.3ex\hbox{$\sqcap$}}$%
\sqcup$}}
\def\findem{\ifmmode\sqw\else{\ifhmode\unskip\fi\nobreak\hfil
\penalty50\hskip1em\null\nobreak\hfil\sqw
\parfillskip=0pt\finalhyphendemerits=0\endgraf}\fi}

\title{Growth of solutions to NLS on irrational tori}

\author{Yu Deng and Pierre Germain}

\thanks{P. G. is partially supported by NSF grant DMS-1501019.}

\begin{document}

\maketitle

\begin{abstract} We prove polynomial bounds on the $H^s$ growth for the nonlinear Schr\"odinger equation set on a torus, in dimension 3, with super-cubic and sub-quintic nonlinearity. Due to improved Strichartz estimates, these bounds are better for irrational tori than they are for rational tori.
\end{abstract}

\section{Introduction}

\subsection{Growth of Sobolev norms for NLS}

Consider the nonlinear Schr\"odinger equation $(i\partial_t+\Delta)u=|u|^{p-1}u$ set on a torus. In subcritical cases ($p<\frac{d+2}{d-2}$), and for the torus $\mathbb{T}^d = \mathbb{R}^d / \mathbb{Z}^d$, the global existence of finite energy ($H^1$) solutions is known since the foundational work of Bourgain~\cite{B1}. Furthermore, the equation propagates regularity: if the data is smoother, say in $H^s$ for $s>1$, then $u(t) \in H^s$ for all $t$.

The next question is to understand the qualitative behavior of the solution, and first of all: how fast may the $H^s$ norm grow? This question is related to the phenomenon of \emph{weak turbulence} which is generally described as the solution transferring energy to higher and higher frequencies, causing the $H^s$ norm to grow while the $H^1$ norm remains bounded. The growth rate of $H^s$ norm can be seen as a control of how fast this energy transfer is happening.

In general, it is easy to obtain an exponential upper bound for the $H^s$ norm, by iterating local in time theory. In \cite{B2}, using his ``high-low method'', Bourgain was first able to improve this to a bound that is polynomial in time, in the case of a cubic nonlinearity in $2D$ and $3D$. Further improvements, and extension to other dimensions and nonlinearities have since been made in ~\cite{CKO, Soh1, Soh2, Staf}; see also ~\cite{Del, GG, PTV, Soh3, Th,Z} and the references therein for other dispersive models.

The recent work of Bourgain and Demeter~\cite{BD}, which proves optimal Strichartz estimates for the linear Schr\"odinger equation on irrational tori, opened the door to new questions for the nonlinear problem. It also enabled the authors, together with Guth~\cite{DGG}, to show that irrational tori enjoy better Strichartz estimates on long time intervals than rational tori.

The aim of the present paper is to show that these linear estimates can be used to obtain improved $H^s$ growth bounds for nonlinear Schr\"{o}dinger equations on irrational tori, compared with rational tori (see also the recent work \cite{FSWW}, where the authors study the nonlinear Schr\"{o}dinger equation on irrational tori from a different aspect). More precisely, we will consider the nonlinear Schr\"iodinger equation, in dimension $d=3$, with power $3<p<5$.

\subsection{Main results}
Consider a $3D$ super-cubic sub-quintic, defocusing nonlinear Schr\"{o}dinger equation,
\begin{equation}\label{nls0}(i\partial_t+\Delta)u=|u|^{p-1}u,\quad 3<p<5,\end{equation} on $\mathbb{R}\times\mathbb{T}_{\ell}^3$, where $\mathbb{T}_{\ell}^3$ is a rectangular torus\[\mathbb{T}_\ell^3=[0,\ell_1]\times[0,\ell_2]\times[0,\ell_3],\quad \ell=(\ell_1,\ell_2,\ell_3).\] One can prove that (\ref{nls0}) is globally well-posed in $H^1(\mathbb{T}_{\ell}^3)$ (see Proposition \ref{locale} below) with conserved energy\[E_{\ell}[u]= \int_{\mathbb{T}_{\ell}^3}\left(\frac{1}{2}|\nabla u|^2+\frac{1}{p+1}|u|^{p+1}\right)\,\mathrm{d}x.\] Now, consider a solution $u$ to (\ref{nls0}) such that $u(0)\in H^2(\mathbb{T}_{\ell}^3)$; by preservation of regularity one can show that $u(t)\in H^2$ for all time (see Proposition \ref{locale}). We are interested in controlling the possible growth of the $H^2$ norm of $u$. We will prove the following: \begin{thm}\label{main} Suppose $u$ is a solution to (\ref{nls0}) with energy $E_{\ell}[u]=E$ and $\|u(0)\|_{H^2(\mathbb{T}_{\ell}^3)}=A$. Then we have the following.

(1) For any choice of $\ell$, we have \begin{equation}\label{growth}\|u(t)\|_{H^2(\mathbb{T}_{\ell}^3)}\lesssim A+(1+|t|)^{\frac{2}{5-p}+\delta}\end{equation} for any time $t$, and any $\delta>0$. Here and below all implicit constants will depend on $\ell$, $E$, $p$ and $\delta$, but not on $A$ or $t$.

(2) For generic choice of $\ell$, i.e. excluding a subset of $(\mathbb{R}^+)^3$ with measure $0$, we have \begin{equation}\label{growth2}\|u(t)\|_{H^2(\mathbb{T}_{\ell}^3)}\lesssim A+(1+|t|)^{\frac{2}{5-p+\theta(p)}}\end{equation} for any time $t$, where \begin{equation}\label{deftheta}\theta(p)=\frac{\min(p-3,5-p)}{182},\end{equation} which is positive for $3<p<5$.
\end{thm}
\begin{rem} Our choice of $\theta(p)$ is clearly far from optimal, and the result is easily extended to $1< s\leq 2$ (and to $s>2$ if one regularizes the nonlinearity). The point here is that generic irrational tori enjoy strictly better estimates, in terms of growth of higher Sobolev norms of solutions to nonlinear Schr\"{o}dinger equations, than rational ones. 
\end{rem}
\begin{rem}Theorem \ref{main} is actually true assuming some weaker Diophantine condition for $\ell$, for example when\[\left|\frac{n_1}{\ell_1^2}+\frac{n_2}{\ell_2^2}+\frac{n_3}{\ell_3^2}\right|\geq C^{-1}(|n_1|+|n_2|+|n_3|)^{-4}\] for integers $n_1$, $n_2$, $n_3$, not all zero (for generic $\ell$ the exponent $4$ can be replaced by $2+\delta$); in particular Theorem \ref{main} is true for $\ell=(1,\sqrt[4]{2},\sqrt[4]{3})$.
\end{rem}

\subsection{Idea of the proof}
We describe here the main idea of the paper; for relevant notations and spaces see Sections \ref{notat} and \ref{seclin}.

Bourgain's original proof of polynomial growth uses a high-low decomposition; here we shall use a variation of this idea that is technically more convenient, namely the (upside down) \emph{I-method} developed by Colliander, Keel, Staffilani, Takaoka and Tao. For a reference on $I$-method, see for example \cite{CKSTT}.

In order to prove an upper bound \[\|u(t)\|_{H^2}\lesssim\max(\|u(0)\|_{H^2},(1+|t|)^{1/\theta}),\] it suffices to show that\begin{equation}\label{need}\text{If $\|u(t)\|_{H^2}\lesssim N$, then $\|u(t')\|_{H^2}\lesssim N$ for $t\leq t'\leq t+N^{\theta}$.}\end{equation} Fix a scale $N$, define a multiplier $\mathcal{D}$ (see (\ref{mult}) below) such that $\mathcal{D}=1$ for frequencies $\ll N$ and $\mathcal{D}=|\nabla|$ for frequencies $\gtrsim N$, then using conservation of energy, we see that $\|u(t)\|_{H^2}\lesssim N$ if and only if $E_{\ell}[\mathcal{D}u]\lesssim 1$.

 The idea of the $I$-method is then to control the increment of the energy $E_{\ell}[\mathcal{D}u]$. By the structure of the equation (\ref{nls0}), one has that\footnote{Strictly speaking $(\mathcal{D}u)^{p-1}$ should be a $p-1$-homogeneous function of $\mathcal{D}u$, but these are just trivial differences.}\[E_{\ell}[\mathcal{D}u(t_2)]-E_{\ell}[\mathcal{D}u(t_1)]\sim\int_{t_1}^{t_2}\int_{\mathbb{T}_{\ell}^3}(\nabla\mathcal{D}u)^2(\mathcal{D}u)^{p-1}\,\mathrm{d}x\mathrm{d}t.\] Moreover,  the integrand above vanishes (or asymptotically vanishes) if the frequencies involved in $(\mathcal{D}u)^{p-1}$ are all $\ll N$, since since $\mathcal{D}=1$ for frequencies $\lesssim N$ and that $E_{\ell}[u]$ is conserved for (\ref{nls0}).

Therefore, one actually has that \begin{equation}\label{incre0}E_{\ell}[\mathcal{D}u(t_2)]-E_{\ell}[\mathcal{D}u(t_1)]\sim\int_{t_1}^{t_2}\int_{\mathbb{T}_{\ell}^3}(\nabla\mathcal{D}u)^2(P_{\gtrsim N}\mathcal{D}u)(\mathcal{D}u)^{p-2}\,\mathrm{d}x\mathrm{d}t.\end{equation} Suppose $E_{\ell}[\mathcal{D}u(t)]$ is bounded, then by local theory one can bound $\mathcal{D}u$ in $X^{1,1/2+}$ locally (see Section \ref{notat} for notations and (\ref{xsb}) for the definition of $X^{s,b}$ spaces). By H\"{o}lder one has that\[|E_{\ell}[\mathcal{D}u(t_2)]-E_{\ell}[\mathcal{D}u(t_1)]|\lesssim\|\nabla\mathcal{D}u\|_{L_{t,x}^{10/3-}}^2\|P_{\gtrsim N}\mathcal{D}u\|_{L_{t,x}^{10/(6-p)+}}\|\mathcal{D}u\|_{L_{t,x}^{10-}}^{p-2},\] and by Strichartz one has \[\|\nabla\mathcal{D}u\|_{L_{t,x}^{10/3-}}\lesssim \|\mathcal{D}u\|_{X^{1,1/2+}},\quad \|\mathcal{D}u\|_{L_{t,x}^{10-}}\lesssim \|\mathcal{D}u\|_{X^{1,1/2+}}; \quad\|\mathcal{D}u\|_{L_{t,x}^{10/(6-p)+}}\lesssim N^{(p-5)/2+}\|\mathcal{D}u\|_{X^{1,1/2+}}\] on time intervals of length $1$, which implies that \begin{equation}\label{bootrat}\text{If $E_{\ell}[\mathcal{D}u(t_1)]\lesssim1$, then $|E_{\ell}[\mathcal{D}u(t_2)]-E_{\ell}[\mathcal{D}u(t_1)]|\lesssim N^{(p-5)/2+}$, when $t_2-t_1\leq 1$.}\end{equation} Note that the gain $N^{(p-5)/2}$ relies precisely on the subcritical nature of (\ref{nls0}) when $p<5$. By iteration, this then implies (\ref{need}) for $\theta=(5-p)/2-$.

The above is what happens for rational tori; for (generic) irrational tori, one can resort to the long-time Strichartz estimate, which is proved in \cite{DGG}:\[\|e^{it\Delta}P_{\gtrsim N}\mathcal{D}f\|_{L_{t,x}^{10/(6-p)+}}\lesssim N^{(p-5)/2+}\|\mathcal{D}f\|_{H^1},\] on an interval of length $N^{\gamma}$, where $\gamma$ is sufficiently small. This in particular implies the bound\[\sum_{|m|\leq N^\gamma}\|e^{it\Delta}P_{\gtrsim N}\mathcal{D}u(t_1)\|_{L_{t,x}^{10/(6-p)+}([m,m+1])}\lesssim N^{(p-5)/2+\gamma+}N^{-\gamma(6-p)/10}\] for the linear solution $e^{it\Delta}\mathcal{D}u(t_1)$; one can then show that the same bound is true for the nonlinear solution $\mathcal{D}u$ also, again thanks to the subcritical nature of (\ref{nls0}). Plugging this into (\ref{incre0}), we get that \begin{equation}\label{bootirat}\text{If $E_{\ell}[\mathcal{D}u(t_1)]\lesssim1$, then $|E_{\ell}[\mathcal{D}u(t_2)]-E_{\ell}[\mathcal{D}u(t_1)]|\lesssim N^{(p-5)/2+\gamma+}N^{-\gamma(6-p)/10}$, when $t_2-t_1\leq N^{\gamma}$,}\end{equation} in place of (\ref{bootrat}). Iterating (\ref{bootirat}) we get that (\ref{need}) holds for \[\theta=\gamma+\frac{5-p}{2}-\gamma+\frac{(6-p)\gamma-}{10}=\frac{5-p}{2}+\frac{(6-p)\gamma}{10}-,\] which improves upon the rational case.
\subsection{Notations}\label{notat}
For a function $f$ on $\mathbb{R} \times \mathbb{T}^3$, and $(\tau,k) \in \mathbb{R} \times \mathbb{Z}^3$, let
$$
\widehat{f}(\tau,k) = \int_{\mathbb{R} \times \mathbb{T}^3} f(t,x) e^{-2\pi i (\tau t + k \cdot x)}\,\mathrm{d}x\,\mathrm{d}t.
$$ 
For $N$ a dyadic number, let $P_N$, $P_{<N}$ be the standard Littlewood-Paley projections; moreover for any set $B\subset\mathbb{R}^3$ which is a ball, an annulus, a cube or a rectangular cuboid, we shall define the projection $P_B$ in a similar way as $P_N$. For any $r\in\mathbb{R}$, we say a function $F(z):\mathbb{C}\mapsto\mathbb{C}$ is of type $r$, if\[\left|\partial_z^m\partial_{\overline{z}}^nF\right|\lesssim_{m,n} |z|^{r-m-n}\] for all $m,n\geq 0$; we denote by $F_r$ a general function of type $r$. For example, $|z|^rz^m(\overline{z})^n$ is of type $r+m+n$ for any $r\in\mathbb{R}$ and $m,n\in\mathbb{Z}$.

For any fixed scale $N$, define the multiplier $\mathcal{D}$ to be
\begin{equation}\label{mult}\widehat{\mathcal{D}u}(k)=m(k)\widehat{u}(k),\quad m(k)=\theta(k/N),\end{equation} where $\theta=\theta(y)$ is a smooth even function such that $\theta(y)=1$ for $|y|\leq 1$ and $\theta(y)=|y|$ for $|y|\geq 2$.

We will use $\chi$ to denote general cutoff functions: compactly supported, and equal to one in a neighborhood of zero. 

We write $O(1)$ for a constant, and $A \lesssim B$ if there exists a constant $C$ such that $A \leq CB$. We will use $o(1)$ to denote any quantity that can be chosen arbitrarily small, and denote by $a+$ (or $a-$) anything larger (or smaller) than $a$ that is $o(1)$ close to $a$.

\section{Preparations} 

\subsection{Change of variables} First note that, by a change of variables, one can reduce (\ref{nls0}) to the equation \begin{equation}\label{nls}(i\partial_t+\Delta_{\beta})u=|u|^{p-1}u,\end{equation} on $\mathbb{R}\times \mathbb{T}^3$, where $\mathbb{T}^3=[0,1]^3$ is the standard square torus, and $\Delta_{\beta}$ is the ``anisotropic'' Laplacian\begin{equation}\label{rect}\Delta_{\beta}=\beta_1\partial_{x_1}^2+\beta_2\partial_{x_2}^2+\beta_3\partial_{x_3}^2,\quad\beta_i=\ell_i^{-2}.\end{equation} Note that the mapping from $\ell$ to $\beta=(\beta_i)$ preserves zero measure sets, thus preserves genericity. The corresponding conserved energy for (\ref{nls}) is \begin{equation}\label{conserve}\qquad E_{\beta}[u]=\int_{\mathbb{T}^3}\left(\frac{1}{2}\sum_{i=1}^3\beta_i|\partial_i u|^2+\frac{1}{p+1}|u|^{p+1}\right)\,\mathrm{d}x,\end{equation}  which, for simplicity, will be written as $E[u]$ from now on.
\subsection{Linear estimates}\label{seclin} Recall the definition of $X^{s,b}$ and (for an interval $I$ of $\mathbb{R}$) $X^{s,b,I}$ spaces
\begin{equation}\label{xsb}\|u\|_{X^{s,b}}^2=\sum_{k\in\mathbb{Z}^3}\int_{\mathbb{R}}\langle k\rangle^{2s}\langle\tau+ 2\pi Q(k)\rangle^{2b}|\widehat{u}(k,\tau)|^2\,\mathrm{d}\tau,\end{equation}
\begin{equation}\label{xsbi}\|u\|_{X^{s,b,I}}=\inf_{g:\,g\equiv f\,\text{on}\,I}\|g\|_{X^{s,b}},\end{equation} where $Q(k)=\beta_1k_1^2+\beta_2k_2^2+\beta_3k_3^2$. We have the following linear estimates.
\begin{prop} \label{linear} Let $\chi$ be a smooth cutoff. Parts (1)$\sim$(5) below hold for all $\beta=(\beta_i)$, and part (6) holds for generic $\beta_i$.

(1) For all $s,b\in\mathbb{R}$, one has \begin{equation}\label{lin1}\|\chi(t)u\|_{X^{s,b}}\lesssim\|u\|_{X^{s,b}},\quad \|\chi(t)e^{it\Delta_{\beta}}f\|_{X^{s,b}}\lesssim\|f\|_{H^s};\end{equation}

(2) For $\varepsilon<1$ and $-1/2<b\leq b'<1/2$, one has \begin{equation}\label{lin2}\|\chi(\varepsilon^{-1}t)u\|_{X^{s,b}}\lesssim\varepsilon^{b'-b}\|u\|_{X^{s,b'}};\end{equation}

(3) For $1/2<b<1$, one has \begin{equation}\label{lin3}\left\|\chi(t)\int_0^t e^{i(t-t')\Delta_{\beta}}(N(t'))\,\mathrm{d}t'\right\|_{X^{s,b}}\lesssim\|N\|_{X^{s,b-1}};\end{equation}

(4) For $q\geq 10/3$ and $b>1/2$, one has that
\begin{equation}\|P_{N}u\|_{L_{t,x}^q([0,1]\times\mathbb{T}^3)}\lesssim N^{\frac{3}{2}-\frac{5}{q}+}\|P_Nu\|_{X^{0,b}}.\end{equation} Moreover, the same bound holds if one replaces $P_Nu$ by $P_{B}u$, where $B\subset\mathbb{R}^3$ any cube of size $N$.

(5) For $q\geq 10/3$ and $b>1/2$, one has that \begin{equation}\|P_{R}u\|_{L_{t,x}^q([0,1]\times\mathbb{T}^3)}\lesssim N^{\frac{3}{2}-\frac{5}{q}+}\left(\frac{M}{N}\right)^{\frac{1}{2}-\frac{5}{3q}}\|P_Ru\|_{X^{0,b}},\end{equation} where $R\subset\mathbb{R}^3$ is any rectangular cuboid of dimensions $N\times N\times M$ with $M\leq N$.

(6) For generic $\beta=(\beta_i)$, and $q>10/3$ and $b>1/2$, one has that \begin{equation}\|P_{N}u\|_{L_{t,x}^q([0,N^{\nu(q)}]\times\mathbb{T}^3)}\lesssim N^{\frac{3}{2}-\frac{5}{q}+}\|P_Nu\|_{X^{0,b}},\end{equation} where \[\nu(q)=\left\{\begin{aligned}
&\frac{4(3q-10)}{3q+14},&q<6,\\
&4,&q\geq 6.\end{aligned}\right.\]
\end{prop}
\begin{proof} Parts (1)$\sim$(3) are well-known, see~\cite{Tao}; part (4) follows from the full Strichartz estimate of Bourgain-Demeter~\cite{BD}. The corresponding result for $P_Bu$ follows from Galilean invariance. Part (6) is proved in ~\cite{DGG}. Finally, part (5) follows from part (4) in the case $q=10/3$, from Hausdorff-Young and H\"{o}lder in the case $q=\infty$, and from interpolation for any $q$ in between.
\end{proof}

\begin{rem}\label{int} By interpolating (4) with the trivial bounds\[\|P_Nu\|_{L_{t,x}^2}=\|P_Nu\|_{X^{0,0}},\qquad \|P_{N}u\|_{L_{t,x}^{\infty}}\lesssim N^{\frac{3}{2}}\|P_Nu\|_{X^{0,1/2+}}\] and by duality, one gets a number of Strichartz estimates that will be used below; for example\[\|P_{N}u\|_{L_{t,x}^{10/3-}([0,1]\times\mathbb{T}^3)}\lesssim N^{o(1)}\|P_Nu\|_{X^{0,1/2-}}\] follows from interpolating the corresponding $X^{0,1/2+}$ and $X^{0,0}$ estimates. Moreover, by summing over $N$ one also gets estimates such as 
\begin{align*}
& \|u\|_{L_{t,x}^{10-}([0,1]\times\mathbb{T}^3)}\lesssim\|u\|_{X^{1,1/2+}}.
\end{align*}
\end{rem}

\subsection{A nonlinear lemma} For nonlinear terms of the type $F(u)$, with $F$ a function of type $r$, one cannot employ the paraproduct decomposition to obtain estimates on dyadic frequency blocks. The following lemma circumvents this difficulty.
 
\begin{prop}[A nonlinear lemma]\label{nonlem} Let $F$ be any function of type $r$, where $r\geq 1$, then for any dyadic $K$ we have \begin{equation}\label{nonlemest}\|P_KF(u)\|_{L_{t,x}^{q_0}}\lesssim K^{o(1)}\|u\|_{L_{t,x}^{q_1}}^{r-1}\cdot\sum_{M}\min(1,K^{-1}M)\|P_Mu\|_{L_{t,x}^{q_2}},\end{equation} provided that \[q_0,q_1,q_2\in [1,\infty],\quad \frac{1}{q_0}=\frac{r-1}{q_1}+\frac{1}{q_2}.\]
\end{prop} 
\begin{proof} We decompose\[P_KF(u)=\sum_{M}P_K\left[F(P_{\leq M}u)-F(P_{\leq M/2}u)\right]:=\sum_M P_K\mathcal{I}_M,\] and notice that \[\mathcal{I}_M=P_{M}u\cdot \int_0^1(\partial_zF)(P_{\leq M/2}u+\theta P_{M}u)\,\mathrm{d}\theta+\overline{P_{M }u}\cdot \int_0^1(\partial_{\overline{z}}F)(P_{\leq M/2}u+\theta P_{M}u)\,\mathrm{d}\theta.\] Therefore, when $M\geq K$ we can estimate\[\|P_K\mathcal{I}_M\|_{L_{t,x}^{q_0}}\lesssim\|P_{M}u\|_{L_{t,x}^{q_2}}\left(\|P_{\leq M/2}u\|_{L_{t,x}^{q_1}}+\|P_{M}u\|_{L_{t,x}^{q_1}}\right)^{r-1}\lesssim \|P_{M}u\|_{L_{t,x}^{q_2}}\|u\|_{L_{t,x}^{q_1}}^{r-1},\] while for $M\leq K$ we have\[\begin{split}\|P_K\mathcal{I}_M\|_{L_{t,x}^{q_0}}&\lesssim K^{-1}\left\|\nabla \left[F(P_{\leq M}u)-F(P_{\leq M/2}u)\right]\right\|_{L_{t,x}^{q_0}}\\
&\lesssim K^{-1}\left(\left\||\nabla P_{\leq M}u|\cdot|P_{\leq M}u|^{r-1}\right\|_{L_{t,x}^{q_0}}+\left\||\nabla P_{\leq M/2}u|\cdot|P_{\leq M/2}u|^{r-1}\right\|_{L_{t,x}^{q_0}}\right)\\
&\lesssim K^{-1}\|\nabla P_{\leq M}u\|_{L_{t,x}^{q_2}}\|P_{\leq M}u\|_{L_{t,x}^{q_1}}^{r-1}\\
&\lesssim K^{-1}\|u\|_{L_{t,x}^{q_1}}^{r-1}\sum_{M'\leq M}M'\|P_{M'}u\|_{L_{t,x}^{q_2}}.\end{split}\] Summing over $M$, this implies (\ref{nonlemest}).
\end{proof}
\begin{prop}\label{propertyd} Recall the multiplier $\mathcal{D}$ defined in (\ref{mult}). We have \begin{equation}\label{product}\|\mathcal{D}(fg)\|_{L_{t,x}^{q_0}}\lesssim \|\mathcal{D}f\|_{L_{t,x}^{q_1}}\|\mathcal{D}g\|_{L_{t,x}^{q_2}},\end{equation} provided \[q_0,q_1,q_2\in[1,\infty],\quad \frac{1}{q_0}=\frac{1}{q_1}+\frac{1}{q_2},\] and \begin{equation}\label{nonlinear}\|\mathcal{D}F(u)\|_{L_{t,x}^q}\lesssim\|\mathcal{D}u\|_{L_{t,x}^{qr}}^r\end{equation} provided that $F$ is of type $r$, where $\min(q,r)\geq 1$. Moreover, we have \begin{equation}\label{nonlemd}\|P_K\mathcal{D}F(u)\|_{L_{t,x}^{q_0}}\lesssim K^{o(1)}\|\mathcal{D}u\|_{L_{t,x}^{q_1}}^{r-1}\cdot\sum_{M}\min(1,K^{-1}M)\|P_M\mathcal{D}u\|_{L_{t,x}^{q_2}},\end{equation} provided that $F$ is of type $r\geq 2$, and \[q_0,q_1,q_2\in [1,\infty],\quad \frac{1}{q_0}=\frac{r-1}{q_1}+\frac{1}{q_2}.\]
\end{prop}
\begin{proof} Note that \[\|\mathcal{D}u\|_{L_{t,x}^{p}}\sim\|u\|_{L_{t,x}^p}+N^{-1}\|\nabla u\|_{L_{t,x}^p},\] from which (\ref{product}) and (\ref{nonlinear}) follow easily. As for (\ref{nonlemd}), we repeat the proof of Proposition \ref{nonlem}, and write\[P_K\mathcal{D}F(u)=\sum_{M}P_K\mathcal{DI}_M,\] where \[\mathcal{I}_M=F(P_{\leq M}u)-F(P_{\leq M/2}u).\] By the same arguments in the proof of Proposition \ref{nonlem}, together with (\ref{product}) and (\ref{nonlinear}), and using the fact that $\nabla F$ is or type $r-1\geq 1$, one gets that\[\|P_K\mathcal{DI}_M\|_{L_{t,x}^{q_0}}\lesssim \|\mathcal{D}u\|_{L_{t,x}^{q_1}}^{r-1}\|P_{M}\mathcal{D}u\|_{L_{t,x}^{q_2}}\] if $M\geq K$, and that\[\|P_K\mathcal{DI}_M\|_{L_{t,x}^{q_0}}\lesssim K^{-1}\|\mathcal{D}u\|_{L_{t,x}^{q_1}}^{r-1}\sum_{M'\leq M}M'\|P_{M'}\mathcal{D}u\|_{L_{t,x}^{q_2}}\] when $M\leq K$. Summing over $M$ gives (\ref{nonlemd}).
\end{proof}
\section{Local theory}
\begin{prop} Fix $b=1/2+$.
\begin{itemize}
\item[(1)] (Local well-posedness) \label{locale} Suppose $\|f\|_{H^1}\leq E$, then for a short time $\varepsilon=\varepsilon(E)\ll 1$, the equation (\ref{nls}) has a unique solution $u\in X^{1,b,[-\varepsilon,\varepsilon]}$ with initial data $u(0)=f$, and one has \[\|u\|_{X^{1,b,[-\varepsilon,\varepsilon]}}\lesssim_E1.\] 
\item[(2)] (Propagation of regularity) Moreover, if in addition $\|f\|_{H^2}\leq A$, then we also have \[\|u\|_{X^{2,b,[-\varepsilon,\varepsilon]}}\lesssim_EA.\]
\end{itemize}
\end{prop}
\begin{proof} (1) Fix $b'=b+$. For a suitable cutoff function $\chi$, we only need to prove that the mapping \[u\mapsto\mathcal{N}u:=\chi(t)e^{it\Delta_\beta}f-i\chi(t)\int_0^te^{i(t-t')\Delta_\beta}\left(\chi(\varepsilon^{-1}t')|u(t')|^{p-1}u(t')\right)\,\mathrm{d}t'\] is a contraction mapping from the set\[K:=\left\{u\in X^{1,b}:\|u-\chi(t)e^{it\Delta_\beta}f\|_{X^{1,b}}\leq 1\right\}\] to itself, since this would imply that the Duhamel map $\mathcal{N}$  is also a contraction mapping from the $1$-neighborhood of $e^{it\Delta_\beta}f$ in $X^{1,b,[-\varepsilon,\varepsilon]}$ to itself.

Now suppose $u\in K$, then \[\|u\|_{X^{1,b}}\lesssim E\] by (\ref{lin1}), and by (\ref{lin1})$\sim$(\ref{lin3}) together with the definition of $\mathcal{N}$, we get that
\begin{equation}\label{boot2}\left\|\mathcal{N}u-\chi(t)e^{it\Delta_\beta}f\right\|_{X^{1,b}}\lesssim \varepsilon^{b'-b}\left\|\chi(t)|u|^{p-1}u\right\|_{X^{1,b'-1}}.\end{equation} Thus, if we can prove that
\begin{equation}\label{nonest}\left\|\chi(t)\nabla(|u|^{p-1}u)\right\|_{X^{0,b'-1}}\lesssim E^p,\end{equation} then (\ref{boot2}) would imply that \[\left\|\mathcal{N}u-\chi(t)e^{it\Delta_\beta}f\right\|_{X^{1,b}}\lesssim\varepsilon^{b'-b}E^{p},\] which gives that $\mathcal{N}u\in K$ when $\varepsilon$ is small enough (strictly speaking one has to bound the $X^{0,b'-1}$ norm of $\chi(t)|u|^{p-1}u$ also, but that easily follows from the proof below). Now let us prove (\ref{nonest}). Note that\[\nabla(|u|^{p-1}u)=F_{p-1}(u)\nabla u+F_{p-1}(u)\nabla\overline{u},\] where $F_{p-1}(u)$ denotes a function of type $p-1$; we will only prove the bound for $J:=F_{p-1}(u)\nabla \overline{u}$, since the other term is similar.

Write $I:=F_{p-1}(u)$. Let $N_2$ and $N_3$ be dyadic scales and $\max (N_2,N_3)=N$, we decompose\[J=\sum_{N_2,N_3}P_{N_2}I\cdot P_{N_3}\nabla\overline{u}.\] By Proposition \ref{nonlem} and Strichartz we know that\begin{multline}\label{boundn2}\|P_{N_2}I\|_{L_{t,x}^{(5/2)+}}\lesssim N_2^{o(1)}\|u\|_{L_{t,x}^{10-}}^{p-2}\sum_{M}\min(1,N_2^{-1}M)\|P_Mu\|_{L_{t,x}^{[10/(6-p)]+}}\\\lesssim N_2^{o(1)}E^{p-2}\sum_{M}\min(1,N_2^{-1}M)M^{\frac{p-5}{2}+}E\lesssim N_2^{\frac{p-5}{2}+}E^{p-1}.\end{multline} Thus if $N_2\geq N^{1/2}$, by Strichartz and dual Strichartz we have
\begin{multline}\|\chi(t)P_{N_2}I\cdot P_{N_3}\nabla\overline{u}\|_{X^{0,b'-1}}\lesssim N^{o(1)}\|P_{N_2}I\cdot P_{N_3}\nabla\overline{u}\|_{L_{t,x}^{(10/7)+}}\\\lesssim  N^{o(1)} \|P_{N_2}I\|_{L_{t,x}^{(5/2)+}}\|\nabla P_{N_3}u\|_{L_{t,x}^{10/3}}\lesssim N^{\frac{p-5}{4}+} E^{p},\end{multline}
which gives an acceptable contribution to obtain (\ref{nonest}) since $p-5<0$. Now assume $N_2\leq N^{1/2}$, then $N=N_3$. For fixed $N_3$, we decompose\[P_{N_3}u=\sum_{B\in \mathcal{Q}_{N_2,N_3}}P_{B}u,\] where $\mathcal{Q}_{N_2,N_3}$ is a partition of $\{k\in\mathbb{R}^3:|k|\sim N_3\}$ by cubes of size $N_2$. Therefore we have \[J=\sum_{N_2}\sum_{N_3}\sum_{B\in \mathcal{Q}_{N_2,N_3}}P_{N_2}I\cdot  P_B\nabla\overline{u}.\]  Now if $B\in \mathcal{Q}_{N_2,N_3}$ and $B'\in\mathcal{Q}_{N_2,N_3'}$, and either $N_3\gg N_3'$ or $N_3'\gg N_3$ or $\mathrm{dist}(B,B')\gg N_3$, it is easily seen that the terms $P_{N_2}I\cdot P_B\nabla\overline{u}$ and $P_{N_2}I\cdot P_{B'} \nabla \overline{u}$ must have disjoint Fourier support, and are thus orthogonal. Therefore we have\[\|\chi(t)J\|_{X^{0,b'-1}}\lesssim\sum_{N_2}\bigg(\sum_{N_3}\sum_{B\in \mathcal{Q}_{N_2,N_3}}\|\chi(t)P_{N_2}I\cdot P_B\nabla\overline{u}\|_{X^{0,b'-1}}^2\bigg)^{1/2}.\]
 Moreover for $B \in \mathcal{Q}_{N_2,N_3}$ we actually have \[P_{N_2}I\cdot P_{B}\nabla\overline{u}=P_{10B}(P_{N_2}I\cdot P_{B}\nabla\overline{u}),\]thus by (\ref{boundn2}), Strichartz, dual Strichartz and H\"{o}lder we have
\begin{multline}\|\chi(t)P_{N_2}I\cdot P_{B}\nabla\overline{u}\|_{X^{0,b'-1}}\lesssim N_2^{o(1)} \|P_{10B}(P_{N_2}I\cdot P_{B}\nabla\overline{u})\|_{L_{t,x}^{(10/7)+}}\\\lesssim N_2^{o(1)} \|P_{N_2}I\|_{L_{t,x}^{(5/2)+}}\|\nabla P_{B}u\|_{L_{t,x}^{10/3}}\lesssim N_2^{\frac{p-5}{2}+}E^{p-1}\|P_{B}u\|_{X^{1,b}},\end{multline}
  which gives \[\sum_{N_3}\sum_{B\in\mathcal{Q}_{N_2,N_3}}\|\chi(t)P_{N_2}I\cdot P_B\nabla\overline{u}\|_{X^{0,b'-1}}^2\lesssim N_2^{(p-5)+}E^{2(p-1)}\|u\|_{X^{1,b}}^2.\] Taking square root and summing in $N_2$, we get that \[\|\chi(t)J\|_{X^{0,b'-1}}\lesssim E^{p-1}\|u\|_{X^{1,b}}\lesssim  E^{p},\] and this proves (\ref{nonest}).

\bigskip

To show that $\mathcal{N}$ is a contraction mapping, it suffices to show that \[\left\|\chi(t)\nabla(|u|^{p-1}u-|v|^{p-1}v)\right\|_{X^{0,b'-1}}\lesssim E^{p-1}\|u-v\|_{X^{1,b}}\] provided $\|u\|_{X^{1,b}}+\|v\|_{X^{1,b}}\lesssim E$. This can be done by writing\[\nabla(|u|^{p-1}u-|v|^{p-1}v)=\nabla\bigg((u-v)\cdot\int_0^1F_{p-1}(v+\theta(u-v))\,\mathrm{d}\theta\\+\overline{u-v}\cdot\int_0^1F_{p-1}(v+\theta(u-v))\,\mathrm{d}\theta\bigg).
\] Fix $\theta$ and let $v+\theta(u-v)=w$, we only need to estimate the terms\[F_{p-1}(w)\nabla(u-v),\quad F_{p-1}(w)\nabla(\overline{u-v}),\quad (u-v)F_{p-2}(w)\nabla w,\quad \overline{u-v}\cdot F_{p-2}(w)\nabla w.\] The first two terms can be estimated in exactly the same way as above, and the next two terms can be estimated in the same way as the last one. Now let us prove the estimate for $\overline{u-v}\cdot F_{p-2}(w)\nabla w$. Let $I'=\overline{u-v}\cdot F_{p-2}(w)$, then if can prove \begin{equation}\label{new0}\|P_{N_2}I'\|_{L_{t,x}^{(5/2)+}}\lesssim N_2^{\frac{p-5}{2}+}E^{p-2}\|u-v\|_{X^{1,b}},\end{equation} the same argument as above will apply to give the desired estimate. To prove (\ref{new0}), we perform a similar (and simpler) argument as in the proof of Proposition \ref{nonlem}. Write $I'=I''+R'$, where $I''=P_{\leq N_2}(\overline{u-v}) \cdot F_{p-2}(P_{\leq N_2}w)$, thus\[|\nabla I''|\lesssim|\nabla P_{\leq N_2}(u-v)|\cdot|P_{\leq N_2}w|^{p-2}+|P_{\leq N_2}(u-v)|\cdot|P_{\leq N_2}w|^{p-3}\cdot|\nabla P_{\leq N_2}w|,\] and \[ |R''|\lesssim |P_{>N_2}(u-v)|\cdot|P_{\leq N_2}w|^{p-2}+|P_{>N_2}w|\cdot(|w|+|P_{\leq N_2}w|)^{p-3}|u-v|,\] (which follows from the inequality $|F_{p-2}(x)-F_{p-2}(y)|\lesssim|x-y|\cdot(|x|^{p-3}+|y|^{p-3})$ since $p>3$). Then we can bound \[\begin{split}\|P_{N_2}I''\|_{L_{t,x}^{(5/2)+}}\lesssim N_2^{-1}\|\nabla P_{N_2}I''\|_{L_{t,x}^{(5/2)+}}&\lesssim N_2^{-1}\bigg(\|\nabla P_{\leq N_2}(u-v)\|_{L_{t,x}^{10/3}}\|P_{\leq N_2}w\|_{L_{t,x}^{10(p-2)}}^{p-2}\\
&+\|\nabla P_{\leq N_2}w\|_{L_{t,x}^{10/3}}\|P_{\leq N_2}w\|_{L_{t,x}^{10(p-2)}}^{p-3}\|P_{\leq N_2}(u-v)\|_{L_{t,x}^{10(p-2)}}\bigg)\\
&\lesssim N_2^{\frac{p-5}{2}+}E^{p-1}\|u-v\|_{X^{1,b}},\end{split}\] and \begin{equation*}\begin{split}\|P_{N_2}R'\|_{L_{t,x}^{(5/2)+}}&\lesssim \|P_{>N_2}(u-v)\|_{L_{t,x}^{[10/(6-p)]+}}\|P_{\leq N_2}w\|_{L_{t,x}^{10-}}^{p-2}\\&+\|P_{>N_2}w\|_{L_{t,x}^{[10/(6-p)]+}}\left(\|w\|_{L_{t,x}^{10-}}+\|P_{\leq N_2}w\|_{L_{t,x}^{10-}}\right)^{p-3}\left(\|u-v\|_{L_{t,x}^{10-}}+\|P_{\leq N_2}(u-v)\|_{L_{t,x}^{10-}}\right)\\
&\lesssim N_2^{\frac{p-5}{2}+}E^{p-1}\|u-v\|_{X^{1,b}},\end{split}\end{equation*} which completes the proof that $\mathcal{N}$ is a contraction mapping.

\bigskip

\noindent (2) Assume that $\|f\|_{H^2}\leq A$. Using the same methods as above, we only need to show that $\|u\|_{X^{2,b}}\lesssim A$ and $\|u\|_{X^{1,b}}\lesssim E$ implies\[\left\|\chi(t)\nabla^2(|u|^{p-1}u)\right\|_{X^{0,b'-1}}\lesssim E^{p-1}A.\] Notice that \[\nabla^2(|u|^{p-1}u)=F_{p-1}(u)\nabla^2 u+F_{p-1}(u)\nabla^2\overline{u}+F_{p-2}(u)(\nabla u)^{2}+ F_{p-2}(u)(\nabla u)\cdot(\nabla\overline{u})+F_{p-2}(u)(\nabla\overline{u})^{2},\] where we use the convention that $F_r$ denotes a function of type $r$. The bounds for $F_{p-1}(u)\nabla^2 u$ and $F_{p-1}(u)\nabla^2\overline{u}$ is proved in exactly the same way as above, using $\|u\|_{X^{2,b}}$ to control the $\nabla^2u$ and $\nabla^2\overline{u}$ factors; for the other terms we will only consider $F_{p-2}(u)(\nabla u)\cdot(\nabla\overline{u})$, the rest being similar. Decompose\[F_{p-2}(u)(\nabla u)\cdot(\nabla\overline{u})=\sum_{N_2,N_3,N_4}P_{N_2}F_{p-2}(u)\cdot P_{N_3}\nabla u\cdot P_{N_4}\nabla\overline{u},\] and denote $\max(N_2,N_3,N_4)=N$. Without loss of generality we may assume $N_4\geq N_3$; if $N_4\gtrsim N$ we have\[\begin{split}\mathcal{M}&:=\left\|\chi(t)P_{N_2}F_{p-2}(u)\cdot P_{N_3}\nabla u\cdot P_{N_4}\nabla\overline{u}\right\|_{X^{0,b'-1}}\\&\lesssim N^{o(1)} \left\|P_{N_2}F_{p-2}(u)\cdot P_{N_3}\nabla u\cdot P_{N_4}\nabla\overline{u}\right\|_{L_{t,x}^{10/7+}}\\
&\lesssim N^{o(1)} \|P_{N_2}F_{p-2}(u)\|_{L_{t,x}^{[10/(p-2)]-}}\|P_{N_3}\nabla u\|_{L_{t,x}^{[10/(6-p)]+}}\|P_{N_4}\nabla\overline{u}\|_{L_{t,x}^{10/3}}\\
&\lesssim N^{o(1)} \|u\|_{L_{t,x}^{10-}}^{p-2}\|P_{N_3}\nabla u\|_{L_{t,x}^{[10/(6-p)]+}}\|P_{N_4}\nabla\overline{u}\|_{L_{t,x}^{10/3}}\\
&\lesssim N^{o(1)} E^{p-2}N_3^{1+\frac{p-5}{2}+}E\cdot N_4^{-1+}A\lesssim N^{\frac{p-5}{2}+}E^{p-1}A,\end{split}\]
 and when $N_4\ll N$ we must have $N_2\gtrsim N$, so by Proposition \ref{nonlem} we have \begin{multline}\|P_{N_2}F_{p-2}(u)\|_{L_{t,x}^{10/3+}}\lesssim N_2^{o(1)}\|u\|_{L_{t,x}^{10-}}^{p-3}\sum_{M}\min(1,N_2^{-1}M)\|P_{M}u\|_{L_{t,x}^{[10/(6-p)]+}}\\
\lesssim N_2^{o(1)}E^{p-3}\sum_{M}\min(1,N_2^{-1}M)M^{\frac{p-5}{2}+}E\lesssim N_2^{\frac{p-5}{2}+}E^{p-2},\end{multline}thus \[\begin{split}\mathcal{M}&\lesssim N^{o(1)} \left\|P_{N_2}F_{p-2}(u)\cdot P_{N_3}\nabla u\cdot P_{N_4}\nabla\overline{u}\right\|_{L_{t,x}^{10/7+}}\\
&\lesssim N^{o(1)} \|P_{N_2}F_{p-2}(u)\|_{L_{t,x}^{10/3+}}\|P_{N_3}\nabla u\|_{L_{t,x}^{10-}}\|P_{N_4}\nabla\overline{u}\|_{L_{t,x}^{10/3}}\\
&\lesssim N^{\frac{p-5}{2}+}E^{p-2}\cdot N_3N_{4}^{-1}\|P_{N_3}u\|_{L_{t,x}^{10-}}\|P_{N_4}\nabla^2\overline{u}\|_{L_{t,x}^{10/3}}\\
&\lesssim N^{\frac{p-5}{2}+}E^{p-1}A,\end{split}\] as desired.
\end{proof}
\section{Proof of Theorem \ref{main}: general case}

We shall use the \emph{I-method}. Recall the multiplier $\mathcal{D}$ defined in (\ref{mult}).

\begin{prop}\label{incre} 
 Suppose $\|u(0)\|_{H^1} \leq E$ and $\| \mathcal{D} u(0) \|_{H^1} \leq C_1 E$, for a constant $C_1>0$. Choose $\varepsilon$ as in Proposition~\ref{locale}.  Then we have
\[\big|E[\mathcal{D}u(T)]-E[\mathcal{D}u(0)]\big|\lesssim N^{\max(p-5,-1)+}+N^{o(1)}\sum_{M}M^{o(1)}\min(1,N^{-1}M)\|P_M\mathcal{D}u\|_{L_{t,x}^{[10/(6-p)]+}([0,\varepsilon]\times\mathbb{T}^3)} \] 
for $0\leq T\leq \varepsilon$, with the energy functional $E[u]$ defined in (\ref{conserve}).
\end{prop}

\begin{rem} In the above inequality, we chose to keep on the right-hand side the expression $\|P_M\mathcal{D}u\|_{L_{t,x}^{[10/(6-p)]+}([0,\varepsilon]\times\mathbb{T}^3)}$ instead of estimating it by $M^{\frac{p-5}{2}}$, which would be possible through the $X^{s,b}$ norm derived in the proof below. This will be crucial to get the improvement on irrational tori: indeed, we will prove in the next section that, on irrational tori, a better estimate of $\|P_M\mathcal{D}u\|_{L_{t,x}^{[10/(6-p)]+}}$ becomes available.
\end{rem}

\begin{proof} \underline{Step 1: the modified energy indentity.} From the assumption, we know that (recall that implicit constants depend on $E$)\[\|u(0)\|_{H^1}\lesssim 1,\quad \|u(0)\|_{H^2}\lesssim N.\] By Proposition \ref{locale}, one has that\[\|u\|_{X^{1,b,[-\varepsilon,\varepsilon]}}\lesssim1,\quad \|u\|_{X^{2,b,[-\varepsilon,\varepsilon]}}\lesssim N.\] This gives that\[\|\mathcal{D}u\|_{X^{1,b,[-\varepsilon,\varepsilon]}}\lesssim1.\] By considering a suitable extension of $u$, we may assume $\|\mathcal{D}u\|_{X^{1,b}}\lesssim1$.

Now let us compute the time evolution of $E[\mathcal{D}u]$. In fact, one has that\[(i\partial_t+\Delta_{\beta})(\mathcal{D}u)=|\mathcal{D}u|^{p-1}(\mathcal{D}u)+\mathcal{R},\] where\[\mathcal{R}=\mathcal{D}(|u|^{p-1}u)-|\mathcal{D}u|^{p-1}(\mathcal{D}u).\] Now by conservation of energy for (\ref{nls}), one has that
\begin{equation}\begin{aligned}\partial_tE[\mathcal{D}u]&=\int_{\mathbb{T}^3}\left\{\sum_{i=1}^3\beta_i\Re(\partial_i\overline{\mathcal{D}u}\cdot\partial_i(\mathcal{D}u)_t)+\Re(\overline{\mathcal{D}u}\cdot(\mathcal{D}u)_t)\cdot|\mathcal{D}u|^{p-1}\right\}\,\mathrm{d}x\\
&=\sum_{i=1}^3\beta_i\Im\int_{\mathbb{T}^3}(\partial_i\overline{\mathcal{D}u}\cdot \partial_i\mathcal{R})\,\mathrm{d}x+\Im\int_{\mathbb{T}^3}|\mathcal{D}u|^{p-1}\overline{\mathcal{D}u}\cdot\mathcal{R}\,\mathrm{d}x.\end{aligned}\end{equation}
Thus, upon integrating in $t$, we reduce to estimating the space-time integrals\begin{equation}\label{estimate1}\int_{[0,T]\times\mathbb{T}^3}(\partial_i\overline{\mathcal{D}u}\cdot \partial_i\mathcal{R})\,\mathrm{d}x\mathrm{d}t\end{equation} and \begin{equation}\label{estimate2}\int_{[0,T]\times\mathbb{T}^3}|\mathcal{D}u|^{p-1}\overline{\mathcal{D}u}\cdot\mathcal{R}\,\mathrm{d}x\mathrm{d}t.\end{equation}

\bigskip
\noindent \underline{Step 2: bound for (\ref{estimate2}).}
Let $u_1=P_{\leq N/10}u$ and $u_2=u-u_1=P_{>N/10}u$, note that \[\begin{aligned}\mathcal{R}&=\mathcal{D}(|u|^{p-1}u)-|\mathcal{D}u|^{p-1}(\mathcal{D}u)\\&=\mathcal{D}(|u|^{p-1}u-|u_1|^{p-1}u_1)-\left[|\mathcal{D}u|^{p-1}(\mathcal{D}u)-|\mathcal{D}u_1|^{p-1}(\mathcal{D}u_1)\right]-(1-\mathcal{D})\left(|u_1|^{p-1}u_1\right).\end{aligned}\] For the first term, using the identity\[|u|^{p-1}u-|u_1|^{p-1}u_1=u_2\cdot\int_0^1F_{p-1}(u_1+\theta u_2)\,\mathrm{d}\theta+\overline{u_2}\cdot\int_0^1F_{p-1}(u_1+\theta u_2)\,\mathrm{d}\theta,\] and using (\ref{product}), (\ref{nonlinear}), H\"{o}lder and Strichartz, we can bound the corresponding contribution by 
\[\begin{split}
& \left\||\mathcal{D}u|^{p-1}\overline{\mathcal{D}u}\right\|_{L_{t,x}^{(10/p)-}}\cdot\left\|\mathcal{D}(|u|^{p-1}u-|u_1|^{p-1}u_1)\right\|_{L_{t,x}^{[10/(10-p)]+}}\\
&\quad \lesssim\|\mathcal{D}u\|_{L_{t,x}^{10-}}^{p}\|\mathcal{D}u_2\|_{L_{t,x}^{q_1}}\left(\|\mathcal{D}u\|_{L_{t,x}^{10-}}+\|\mathcal{D}u_1\|_{L_{t,x}^{10-}}\right)^{p-1}\lesssim N^{p-5},\end{split}\] where $q_1=10/(11-2p)+$. The same bound holds for the second term, using the fact that\[\left||\mathcal{D}u|^{p-1}(\mathcal{D}u)-|\mathcal{D}u_1|^{p-1}(\mathcal{D}u_1)\right|\lesssim|\mathcal{D}u_2|\cdot(|\mathcal{D}u|+|\mathcal{D}u_1|)^{p-1}.\] For the last term, notice that \[\begin{split}\left\|(1-\mathcal{D})\left(|u_1|^{p-1}u_1\right)\right\|_{L_{t,x}^{[10/(10-p)]+}}&\lesssim N^{-1}\|\nabla(|u_1|^{p-1}u_1)\|_{L_{t,x}^{[10/(10-p)]+}}\\&\lesssim N^{-1}\|u_1\|_{L_{t,x}^{10-}}^{p-1}\|\nabla u_1\|_{L_{t,x}^{q_1}}\lesssim N^{-1}.\end{split}\] Gathering the estimates, we find that \[(\ref{estimate2})\lesssim N^{\max(p-5,-1)+}.\]

\bigskip

\noindent \underline{Step 3: bound for (\ref{estimate1}).} Note that \[\nabla\mathcal{R}=\mathcal{D}\left[F_{p-1}(u)\nabla u+F_{p-1}(u)\nabla\overline{u}\right]-\left[F_{p-1}(\mathcal{D}u)\nabla(\mathcal{D}u)+F_{p-1}(\mathcal{D}u)\nabla\overline{\mathcal{D}u}\right].\] Since the other term is similar, we only consider the term \[\mathcal{D}(F_{p-1}(u)\nabla u)-F_{p-1}(\mathcal{D}u)\nabla(\mathcal{D}u),\] which can be decomposed as\begin{equation}\label{decomposition}\mathcal{D}[(F_{p-1}(u)-F_{p-1}(u_1))\nabla u]-[F_{p-1}(\mathcal{D}u)-F_{p-1}(\mathcal{D}u_1)]\nabla(\mathcal{D}u)+[\mathcal{D}(\mathcal{H}\nabla u)-\mathcal{H}\nabla\mathcal{D}u],\end{equation} where $\mathcal{H}=F_{p-1}(u_1)$. 
For the first term in (\ref{decomposition}), denote $\mathcal{L}=F_{p-1}(u)-F_{p-1}(u_1)$; note that \[\mathcal{L}=u_2\cdot\int_0^1 F_{p-2}(u_1+\theta u_2)\,\mathrm{d}\theta+\overline{u_2}\cdot\int_0^1F_{p-2}(u_1+\theta u_2)\,\mathrm{d}\theta.\] We see that \[\|\mathcal{D}P_K\mathcal{L}\|_{L_{t,x}^{5/2}}\lesssim \|\mathcal{D}u_2\|_{L_{t,x}^{[10/(6-p)]+}}\left(\|\mathcal{D}u_1\|_{L_{t,x}^{10-}}+\|\mathcal{D}u_2\|_{L_{t,x}^{10-}}\right)^{p-2}\lesssim \sum_{M\geq N/10}\|\mathcal{D}P_Mu\|_{L_{t,x}^{[10/(6-p)]+}}\] if $K\leq N$, and (note that $u_1=P_{\leq N/10}u$ satisfies the same estimates as $u$)\begin{equation}\begin{split}\|\mathcal{D}P_K\mathcal{L}\|_{L_{t,x}^{5/2}}&\lesssim\left\|\mathcal{D}P_KF_{p-1}(u)\right\|_{L_{t,x}^{5/2}}+\left\|\mathcal{D}P_KF_{p-1}(u_1)\right\|_{L_{t,x}^{5/2}}\\&\lesssim K^{o(1)}\|\mathcal{D}u\|_{L_{t,x}^{10-}}^{p-2}\sum_{M}\min(1,K^{-1}M)\|P_M\mathcal{D}u\|_{L_{t,x}^{[10/(6-p)]+}}\\&\lesssim K^{o(1)}\sum_{M}\min(1,K^{-1}M)\|P_M\mathcal{D}u\|_{L_{t,x}^{[10/(6-p)]+}}\end{split}\end{equation} by (\ref{nonlemd}) if $K\geq N$, which gives by summation in $K$ that\begin{equation}\label{rdecay}\sum_K K^{o(1)}\|\mathcal{D}P_K\mathcal{L}\|_{L_{t,x}^{5/2}}\lesssim N^{o(1)}\bigg(\sum_{M\geq N/10}\|\mathcal{D}P_Mu\|_{L_{t,x}^{[10/(6-p)]+}}+\sum_{M}M^{o(1)}\min(1,N^{-1}M)\|P_M\mathcal{D}u\|_{L_{t,x}^{[10/(6-p)]+}}\bigg).\end{equation} Then, the contribution corresponding to this term can be decomposed as\[\int_{[0,T]\times\mathbb{T}^3}\partial_i\overline{\mathcal{D}u}\cdot\mathcal{D}(\mathcal{L}\partial_iu)\,\mathrm{d}x\mathrm{d}t=\sum_{K}\sum_{B} \int_{[0,T]\times\mathbb{T}^3}\partial_iP_{10B}\overline{\mathcal{D}u}\cdot\mathcal{D}(P_K\mathcal{L}\cdot \partial_iP_Bu)\,\mathrm{d}x\mathrm{d}t,\] where for fixed $K$, $B$ runs over some partition of $\mathbb{R}^3$ into cubes of size $K$. By orthogonality, this is bounded by \begin{multline*}\sum_K\sum_B \left\|\partial_iP_{10B}\overline{\mathcal{D}u}\right\|_{L_{t,x}^{10/3}}\left\|\mathcal{D}P_K\mathcal{L}\right\|_{L_{t,x}^{5/2}}\left\|\mathcal{D}\partial_iP_Bu\right\|_{L_{t,x}^{10/3}}\\\lesssim \sum_KK^{o(1)}\|\mathcal{D}P_K\mathcal{L}\|_{L_{t,x}^{5/2}}\bigg(\sum_B\left\|\partial_iP_{10B}\overline{\mathcal{D}u}\right\|_{X^{0,b}}^2\bigg)^{1/2}\bigg(\sum_B \left\|\mathcal{D}\partial_iP_Bu\right\|_{X^{0,b}}^{2}\bigg)^{1/2}\\
\lesssim N^{o(1)}\sum_{M}M^{o(1)}\min(1,N^{-1}M)\|P_M\mathcal{D}u\|_{L_{t,x}^{[10/(6-p)]+}}.\end{multline*} 

By analogous arguments, the same bound can be obtained for the second term in (\ref{decomposition}).

For the last term in (\ref{decomposition}), we use the same trick and decompose \begin{multline}\int_{[0,T]\times\mathbb{T}^3}\partial_i\overline{\mathcal{D}u}\cdot[\mathcal{D}(\mathcal{H}\partial_iu)-\mathcal{H}\partial_i\mathcal{D}u]\,\mathrm{d}x\mathrm{d}t\\=\sum_{K}\sum_{B} \int_{[0,T]\times\mathbb{T}^3}\partial_iP_{10B}\overline{\mathcal{D}u}\cdot[\mathcal{D}(P_K\mathcal{H}\cdot \partial_iP_{B}u)-P_K\mathcal{H}\cdot \partial_i\mathcal{D}P_Bu]\,\mathrm{d}x\mathrm{d}t,\end{multline} where $B$ runs over some partition of $\mathbb{R}^3$ into cubes of size $K$. By (\ref{nonlemd}) and the fact that $\|\mathcal{D}u\|_{L_{t,x}^{10-}}\lesssim 1$, we have\[\|\mathcal{D}P_K\mathcal{H}\|_{L_{t,x}^{5/2}}\lesssim K^{o(1)}\sum_{M\lesssim N}\min(1,K^{-1}M)\|P_M\mathcal{D}u\|_{L_{t,x}^{[10/(6-p)]+}}.\]
 Moreover, by definition of $\mathcal{D}$ we have\[\mathcal{F}[\mathcal{D}(P_K\mathcal{H}\cdot \partial_iP_{B}u)-P_K\mathcal{H}\cdot \partial_i\mathcal{D}P_Bu](k)=\sum_{l+m=k}\bigg[\theta\bigg(\frac{k}{N}\bigg)-\theta\bigg(\frac{m}{N}\bigg)\bigg]\widehat{P_K\mathcal{H}}(l)\widehat{\partial_iP_Bu}(m),\] and the symbol \[\theta\bigg(\frac{k}{N}\bigg)-\theta\bigg(\frac{m}{N}\bigg)=\frac{l}{N}\cdot\int_0^1\nabla \theta\bigg(\frac{m+\gamma l}{N}\bigg)\,\mathrm{d}\gamma,\] thus by Coifman-Meyer theory and transference principle we have\[\left\|\mathcal{D}(P_K\mathcal{H}\cdot \partial_iP_{B}u)-P_K\mathcal{H}\cdot \partial_i\mathcal{D}P_Bu\right\|_{L_{t,x}^{10/7}}\lesssim \min(1,KN^{-1})\|\mathcal{D}P_K\mathcal{H}\|_{L_{t,x}^{5/2}}\|\mathcal{D}\nabla P_Bu\|_{L_{t,x}^{10/3}}.\] 
After summing in $K$ and $B$ and using orthogonality, this gives that \begin{equation}
\begin{split}& \left|\int_{[0,T]\times\mathbb{T}^3}\partial_i\overline{\mathcal{D}u}\cdot[\mathcal{D}(\mathcal{H}\partial_iu)-\mathcal{H}\partial_i\mathcal{D}u]\,\mathrm{d}x\mathrm{d}t\right|\\
&\quad \lesssim\sum_K\sum_B \min(1,KN^{-1})\|\mathcal{D}P_K\mathcal{H}\|_{L_{t,x}^{5/2}}\| \mathcal{D} \nabla P_Bu\|_{L_{t,x}^{10/3}}\|\nabla P_{10B}\overline{\mathcal{D}u}\|_{L_{t,x}^{10/3}}\\
&\quad \lesssim\sum_{K}\min(1,KN^{-1})\| \mathcal{D} P_K\mathcal{H}\|_{L_{t,x}^{5/2}}\bigg(\sum_{B}\|\nabla P_B\mathcal{D}u\|_{X^{0,b}}^2\bigg)^{1/2}\bigg(\sum_{B}\|\nabla P_{10B}\overline{\mathcal{D}u}\|_{X^{0,b}}^2\bigg)^{1/2}\\
&\quad\lesssim\sum_{K}\min(1,KN^{-1})\cdot K^{o(1)}\sum_{M\lesssim N}\min(1,K^{-1}M)\|P_M\mathcal{D}u\|_{L_{t,x}^{[10/(6-p)]+}}\\
&\quad\lesssim N^{o(1)}\sum_{M}M^{o(1)}\min(1,N^{-1}M)\|P_M\mathcal{D}u\|_{L_{t,x}^{[10/(6-p)]+}}.\end{split}\end{equation} This completes the proof.
\end{proof}

\begin{proof}[Proof of Theorem \ref{main} in the general case]
First choose $N = A = \| u(0) \|_{H^2}$ and observe that, by Sobolev embedding,
$$
E[ \mathcal{D} u(0) ] \leq C_0 E
$$
for a constant $C_0$. By Strichartz and Proposition \ref{incre}, we know that
$$
\| P_M \mathcal{D} u \|_{L_{t,x}^{[10/(6-p)]+}([0,\varepsilon]\times\mathbb{T}^3)}\lesssim M^{\frac{p-5}{2}+}.
$$
As long as $T$ is such that
$$
\sup_{0\leq t\leq T}E[\mathcal{D}u(t)]\leq 2 C_0 E \qquad \forall t \in [0,T],
$$
we learn from iterating Proposition \ref{incre} that
$$
\sup_{0\leq t\leq T}E[\mathcal{D}u(t)] - C_0 E \lesssim N^{\max(p-5,-1)}T+N^{o(1)}\sum_MM^{o(1)}\min(1,N^{-1}M)\cdot TM^{\frac{p-5}{2}+}\lesssim N^{\frac{p-5}{2}+}T.
$$
By a bootstrap argument, this gives that $E[\mathcal{D}u(t)]\leq 2C_0 E$, and hence $\|u(t)\|_{H^2}\leq 4C_0 A$, up to a time 
\[
T \sim A^{\frac{5-p}{2}-}.
\] 
This implies immediately (\ref{growth}).
\end{proof}

\section{Proof of Theorem \ref{main}: irrational case}\label{irr}  Let $q_0=10/(6-p)+$ and $\sigma=1/2-5/q_0=(p-5)/2+$ (notice that $-1<\sigma<0$). By Proposition \ref{locale} and Strichartz, we know that if $u$ is a solution to (\ref{nls}) with $\|u(0)\|_{H^1}\leq E$, then\begin{equation}\label{localles}\|P_Ku\|_{L_{t,x}^{q_0}([0,\varepsilon]\times\mathbb{T}^3)}\lesssim K^{\sigma+}\end{equation} for any dyadic $K$. The improvement in the irrational case will be based on an improvement of (\ref{localles}), namely we have the following 
\begin{prop}\label{long} Suppose $\beta=(\beta_i)$ satisfies the long-time Strichartz estimates (Proposition \ref{linear}, part (6)). Under the assumption that $u$ is a solution to (\ref{nls}) and \begin{equation}\label{nrgctrl}\sup_{0\leq t\leq \varepsilon K^{\gamma}}\|\mathcal{D}u(t)\|_{H^1}\lesssim 1,\end{equation} we have
\begin{equation}\label{globalles}\sum_{m=0}^{K^{\gamma}}\|P_K\mathcal{D}u\|_{L_{t,x}^{q_0}([m\varepsilon,(m+1)\varepsilon]\times\mathbb{T}^3)}\lesssim K^{\gamma+\sigma+}\cdot\max(K^{-\frac{\gamma}{q_0}},K^{\gamma+\theta_1}). \end{equation}
\end{prop}
\begin{rem} The trivial bound obtain by iterating (\ref{localles}) would be $K^{\gamma+\sigma+}$. We get an improvement for $\gamma<|\theta_1|$.
\end{rem}
\begin{proof} \underline{Step 1: decomposition of the nonlinear term.} Fix $b=1/2+$. The assumption (\ref{nrgctrl}) and Proposition \ref{locale} give a solution $u$ such that $\| u \|_{X^{1,b,[0,\epsilon]}} +\|\mathcal{D}u\|_{X^{1,b,[0,\varepsilon]}}\lesssim 1$. By considering a suitable extension we may assume $\| u \|_{X^{1,b}} + \|\mathcal{D}u\|_{X^{1,b}}\lesssim 1$, and $u$ is compactly supported in time.

Let \[u_1=P_{\leq K/10}u,\quad u_2=u-u_1=P_{>K/10}u.\] By (\ref{nls}) and Taylor expansion one has that
\begin{multline}\label{nonlincut}
(i\partial_t+\Delta_{\beta})P_K\mathcal{D}u=P_K\mathcal{D}(|u|^{p-1}u)=P_K\mathcal{D}\bigg\{|u_1|^{p-1}u_1+\frac{p+1}{2}|u_1|^{p-1}u_2+\overline{u_2}F_{p-1}(u_1)\\+\int_0^1\left[F_{p-2}(u_1+\theta u_2)u_2^2+F_{p-2}(u_1+\theta u_2)u_2\overline{u_2}+F_{p-2}(u_1+\theta u_2)(\overline{u_2})^2\right] \theta\,\mathrm{d}\theta\bigg\}.
\end{multline} Moreover, we fix a scale $K'=K^{\gamma_0}$ and define \[u_3=P_{\leq K'}u,\quad u_4=P_{(K',K/10]}u=u_1-u_3,\] for some $\gamma_0\in(0,1)$ to be determined later, then we can decompose further \[\frac{p+1}{2}|u_1|^{p-1}=h(t)+\mathbb{P}_{\neq 0}\Omega+u_4\cdot\int_0^1 F_{p-2}\left(u_3+\theta u_4\right)\,\mathrm{d}\theta+\overline{u_4}\cdot\int_0^1 F_{p-2}\left(u_3+\theta u_4\right)\,\mathrm{d}\theta,\] where \[h(t)=\mathbb{P}_0\Omega,\quad\Omega=\frac{p+1}{2}|u_3|^{p-1},\] and $\mathbb{P}_0$ denotes the projection onto the zeroth mode. Therefore we get (notice that $P_K\mathcal{D}u_2=P_K\mathcal{D}u$)
\[(i\partial_t+\Delta_{\beta})P_K\mathcal{D}u=h(t) P_K\mathcal{D}u+ \mathcal{R}\,\] 
where
$$
\mathcal{R} = \mathcal{R}_1 + \mathcal{R}_2 + \mathcal{R}_3 + \mathcal{R}_4
$$
and
\begin{align*}\mathcal{R}_1&=P_K\mathcal{D}\left(|u_1|^{p-1}u_1\right),\\
\mathcal{R}_2&=P_K\mathcal{D}\left(\overline{u_2}\cdot F_{p-1}(u_1)\right),\\
\mathcal{R}_3&=P_K\mathcal{D}\left(u_2\cdot \mathbb{P}_{\neq 0}\Omega\right),\\
\mathcal{R}_4&=P_K\mathcal{D}\left\{u_2u_4\int_0^1F_{p-2}(u_3+\theta u_4)\,\mathrm{d}\theta+u_2\overline{u_4}\int_0^1F_{p-2}(u_3+\theta u_4)\,\mathrm{d}\theta\right\}\\
&+P_K\mathcal{D}\left\{\int_0^1\left[F_{p-2}(u_1+\theta u_2)u_2^2+2F_{p-2}(u_1+\theta u_2)u_2\overline{u_2}+F_{p-2}(u_1+\theta u_2)(\overline{u_2})^2\right] \theta\,\mathrm{d}\theta\right\}.
\end{align*} Let\[v(t)=\exp\left(i\int_0^t h(t')\mathrm{d}t'\right)\cdot P_K\mathcal{D}u(t),\] then we have that
\[(i\partial_{t}+\Delta_{\beta})v=\exp\left(i\int_0^t h(t')\mathrm{d}t'\right)\cdot\mathcal{R}:=\mathcal{R}',\] so \[v(t)=e^{it\Delta_{\beta}}v(0)-i\int_0^t e^{i(t-t')\Delta_{\beta}}\mathcal{R}'(t')\,\mathrm{d}t'\]  for $0\leq t\leq\varepsilon$, which gives that \[\|\chi(t)(v(t)-e^{it\Delta_{\beta}}v(0))\|_{X^{1,b}}\lesssim \|\chi(t)\mathcal{R}'\|_{X^{1,b-1}}.\]

We next proceed to estimate $\mathcal{R}'$. We will denote $\mathcal{R}_j'$ the term in $\mathcal{R}'$ corresponding to $\mathcal{R}_j$.

\bigskip

\noindent \underline{Step 2: estimate of $\mathcal{R}_1'$ and $\mathcal{R}_3'$.} First, we claim that \[\|\chi(t)\mathcal{R}_4'\|_{X^{1,b-1}}\lesssim K\|\mathcal{R}_4'\|_{L_{t,x}^{10/7+}}\lesssim K^{\sigma\gamma_0+}.\] In fact, we only need to consider $\mathcal{R}_4$. For the term $u_2u_4F_{p-2}(u_3+\theta u_4)$ (the other term being similar), one can bound\[K\left\|P_K\mathcal{D}\left(u_2u_4F_{p-2}(u_3+\theta u_4)\right)\right\|_{L_{t,x}^{10/7+}}\lesssim K^{o(1)}\cdot\|K\mathcal{D}u_2\|_{L_{t,x}^{10/3}}\|\mathcal{D}u_4\|_{L_{t,x}^{q_0}}(\|\mathcal{D}u_3\|_{L_{t,x}^{10-}}+\|\mathcal{D}u_4\|_{L_{t,x}^{10-}})^{p-2},\] which is bounded by $K^{\sigma\gamma_0+}$ since $u_4$ has frequency $\geq K^{\gamma_0}$. 

Next, we prove that $\mathcal{R}_1'$ satisfies better estimates; in fact, to bound $\mathcal{R}_1$ we will write \[\nabla \mathcal{R}_1=P_K\mathcal{D}(F_{p-1}(u_1)\nabla u_1),\] and since $u_1$ is supported in frequency $\leq K/10$, we know actually that\[\nabla \mathcal{R}_1=P_K\mathcal{D}(P_{[K/4,K]}F_{p-1}(u_1)\cdot\nabla u_1).\] Thus by (\ref{nonlemd}) we have\[K\|\mathcal{R}_1\|_{L_{t,x}^{10/7+}}\lesssim K^{o(1)}\|\mathcal{D}\nabla u_1\|_{L_{t,x}^{10/3}}\left\|P_{[K/4,K]}\mathcal{D}F_{p-1}(u_1)\right\|_{L_{t,x}^{5/2+}}\lesssim K^{o(1)}\sum_{M}\min(1,K^{-1}M)\|\mathcal{D}P_{M}u_1\|_{L_{t,x}^{q_0}}\] using the fact that $\|\mathcal{D}u\|_{L_{t,x}^{10-}}\lesssim 1$. This implies $K\|\mathcal{R}_1\|_{L_{t,x}^{10/7+}}\lesssim K^{\sigma+}$, since by Strichartz we have $\|\mathcal{D}P_{M}u_1\|_{L_{t,x}^{q_0}}\lesssim M^{\sigma+}$.

\bigskip

\noindent \underline{Step 3: estimate of $\mathcal{R}_3'$.} In $\mathcal{R}_3$ we may replace $\mathbb{P}_{\neq 0}\Omega$ by $\mathbb{P}_{\neq 0}P_{\leq K'}\Omega$ (and $u_2$ by $P_{[K/4,4K]}u$), since\[\|\mathcal{D}P_{>K'}\Omega\|_{L_{t,x}^{5/2+}}\lesssim\sum_{M\geq K'}M^{o(1)}\sum_{L\leq K'}M^{-1}L\|\mathcal{D}P_Lu_3\|_{L_{t,x}^{q_0}}\lesssim \sum_{M\geq K'}M^{\sigma+}\lesssim K^{\sigma\gamma_0+}\] using the fact that $\|\mathcal{D}u_3\|_{L_{t,x}^{10-}}\lesssim 1$, which implies\[K\left\|P_K\mathcal{D}\left(u_2\cdot P_{>K'}\Omega\right)\right\|_{L_{t,x}^{10/7+}}\lesssim \|K\mathcal{D}u_2\|_{L_{t,x}^{10/3}}\|\mathcal{D}P_{>K'}\Omega\|_{L_{t,x}^{5/2+}}\lesssim K^{\sigma\gamma_0+}.\] Let $H(t)=\int_0^{t}h(t')\mathrm{d}t'$, $\Omega'=\mathbb{P}_{\neq 0}P_{\leq K'}\Omega$ and $w=K\max(1,K/N)P_{[K/4,4K]}u$, we have $\|w\|_{X^{0,b}}\lesssim 1$ because $\|\mathcal{D}u\|_{X^{1,b}}\lesssim1$, and $\Omega'$ and $w$ are compactly supported in time. To bound $\mathcal{R}_3'$ we only need to bound \[\min(1,N/K)\|\chi(t)e^{iH(t)}P_K\mathcal{D}(w\cdot \Omega')\|_{X^{0,b-1}}\sim \|\chi(t)e^{iH(t)}P_K(w\cdot \Omega')\|_{X^{0,b-1}},\] by definition of $\mathcal{D}$. Choose some $z$ such that $\|z\|_{X^{0,1-b}}\leq1 $, by duality we only need to bound the quantity \begin{equation}\label{defj}\mathcal{J}:=\int \chi(t)e^{iH(t)}\overline{z}\cdot P_K( w\cdot \Omega')=\sum_{k_1=k_2+k_3}\int_{\xi_1=\xi_0+\xi_2+\xi_3}\widehat{\chi e^{iH}}(\xi_0)\overline{\widehat{P_Kz}(k_1,\xi_1)}\widehat{w}(k_2,\xi_2)\widehat{\Omega'}(k_3,\xi_3).\end{equation} Note that we always have $|k_3|\lesssim K^{\gamma_0}$ in the integral (\ref{defj}). Let $P=\chi e^{iH}$, then clearly $P\in L^2$; moreover $\partial_tP=(ih\chi+\chi')e^{itH}$ also belongs to $L^2$, since $|h(t)|\lesssim\|u(t)\|_{L_x^{p-1}}^{p-1}\lesssim 1$ by Sobolev. This gives by H\"{o}lder that\begin{equation}\label{exp}\|\langle \xi\rangle^{(1/2)-}\widehat{P}(\xi)\|_{L^1}\lesssim 1.\end{equation}

Now, choose $\gamma_1>4\gamma_0+$ to be determined. If the integral (\ref{defj}) is restricted to the region $|\xi_0|\gtrsim K^{\gamma_1}$ by inserting a suitable cutoff function $(1-\chi)(K^{-\gamma_1}\xi_0)$, then using H\"{o}lder, the corresponding contribution will be bounded by\[|\mathcal{J}_1|\lesssim\left\|\mathcal{F}^{-1}\left(\widehat{\chi e^{iH}}(\xi_0)(1-\chi)(K^{-\gamma_1}\xi_0)\right)\right\|_{L_{t}^{\infty}}\cdot \|P_Kz\|_{L_{t,x}^{10/3-}}\|w\|_{L_{t,x}^{10/3}}\|\Omega'\|_{L_{t,x}^{5/2+}}\lesssim K^{-\gamma_1/2+},\] since\[\left\|\widehat{\chi e^{iH}}(\xi_0)(1-\chi)(K^{-\gamma_1}\xi_0)\right\|_{L^1}\lesssim K^{-\gamma_1/2+}\] by (\ref{exp}), and $\|\Omega'\|_{L_{t,x}^{5/2+}}\lesssim\|u_3\|_{L_{t,x}^{5(p-1)/2+}}^{p-1}\lesssim 1$. Now we only need to study \begin{equation}\label{defj2}\mathcal{J}_2:=\sum_{k_1=k_2+k_3}\int_{\xi_1=\xi_0+\xi_2+\xi_3}\widehat{P^*}(\xi_0)\overline{\widehat{P_Kz}(k_1,\xi_1)}\widehat{w}(k_2,\xi_2)\widehat{\Omega'}(k_3,\xi_3),\end{equation} where \[\widehat{P^*}(\xi_0)=\widehat{\chi e^{iH}}(\xi_0)\chi(K^{-\gamma_1}\xi_0),\] and we easily see that\[|P^*(t)|\lesssim(1+|t|)^{-10}.\] Next, if the integral (\ref{defj2}) is restricted to the region $|\xi_1+ 2\pi Q(k_1)|\gtrsim K^{\gamma_1}$ by inserting a cutoff function $(1-\chi)(K^{-\gamma_1}(\xi_1+ 2\pi Q(k_1)))$, then we define $z'$ by \[\widehat{z'}(k_1,\xi_1)=\widehat{P_Kz}(k_1,\xi_1)\cdot (1-\chi)(K^{-\gamma_1}(\xi_1+ 2\pi Q(k_1))),\] and use\[\|z'\|_{L_{t,x}^2}\lesssim K^{-\gamma_1(1-b)}\|z\|_{X^{0,1-b}}\lesssim K^{-\gamma_1/2+}\] to bound the corresponding contribution by \[|\mathcal{J}_3|\lesssim\|P^*\|_{L_{t}^{\infty}}\cdot \|z'\|_{L_{t,x}^{2}}\|w\|_{L_{t,x}^{10/3}}\|\Omega'\|_{L_{t,x}^{5}}\lesssim K^{-\gamma_1/2+}\|P_{\leq K'}u\|_{L_{t,x}^{5(p-1)}}^{p-1}\lesssim K^{-\gamma_1/2+(p-3)\gamma_0/2+}.\] The same estimate holds, with $o(1)$ differences in the power of $K$, if (\ref{defj2}) is restricted to the region $|\xi_2+ 2\pi Q(k_2)|\gtrsim K^{\gamma_1}$. Now we may replace $P_Kz$ by $z'':=P_Kz-z'$, and $w$ by $w''$ which is defined similarly, and reduce to estimating  \begin{equation}\label{defj4}\mathcal{J}_4:=\sum_{k_1=k_2+k_3}\int_{\xi_1=\xi_0+\xi_2+\xi_3}\widehat{P^*}(\xi_0)\overline{\widehat{z''}(k_1,\xi_1)}\widehat{w''}(k_2,\xi_2)\widehat{\Omega'}(k_3,\xi_3).\end{equation} Note that $z''$ and $w''$ still satisfy the $X^{0,1-b}$ and $X^{0,b}$ bounds.

Next, define the operators $\mathcal{P}_L$ and $\mathcal{Q}_L$ as follows:\[\widehat{\mathcal{P}_LF}(k,\xi)=\chi(L^{-1}\xi)\widehat{F}(k,\xi),\qquad\widehat{\mathcal{Q}_LF}(k,\xi)=(1-\chi(L^{-1}\xi))\widehat{F}(k,\xi),\] we decompose $\Omega'=\Omega_1'+\Omega_2'$ where \[\Omega_1'=\frac{p+1}{2}\mathbb{P}_{\neq 0}P_{\leq K'}|\mathcal{P}_{K^{\gamma_1/2}}u_3|^{p-1},\] and \[\Omega_2'=\mathbb{P}_{\neq 0}P_{\leq K'}\bigg(\mathcal{Q}_{K^{\gamma_1/2}} u_3\cdot \int_0^1 F_{p-2}(\mathcal{P}_{K^{\gamma_1/2}}u_3+\theta \mathcal{Q}_{K^{\gamma_1/2}}u_3)\,\mathrm{d}\theta+\overline{\mathcal{Q}_{K^{\gamma_1/2}}u_3}\cdot \int_0^1 F_{p-2}(\mathcal{P}_{K^{\gamma_1/2}}u_3+\theta \mathcal{Q}_{K^{\gamma_1/2}}u_3)\,\mathrm{d}\theta\bigg).\] For $\Omega_2'$ one has \begin{multline}\|\Omega_2'\|_{L_{t,x}^{5/2}([0,1]\times\mathbb{T}^3)}\lesssim\|\mathcal{Q}_{K^{\gamma_1/2}}u_3\|_{L_{t,x}^{5/2}([0,1]\times\mathbb{T}^3)}(\|\mathcal{P}_{K^{\gamma_1/2}}u_3\|_{L_{t,x}^{\infty}([0,1]\times\mathbb{T}^3)}+\|\mathcal{Q}_{K^{\gamma_1/2}}u_3\|_{L_{t,x}^{\infty}([0,1]\times\mathbb{T}^3)})^{p-2}\\\lesssim K^{-\gamma_1/5+(p-2)\gamma_0/2+},\end{multline} since by interpolation\footnote{Namely, by interpolating between $X^{1,0}\hookrightarrow L_{t}^2L_x^{5/2}([0,1]\times\mathbb{T}^3)$ and $X^{1,1/2+}\hookrightarrow L_t^{\infty}L_x^{5/2}([0,1]\times\mathbb{T}^3)$.}\[\|\mathcal{Q}_{K^{\gamma_1/2}}u_3\|_{L_{t,x}^{5/2}([0,1]\times\mathbb{T}^3)}\lesssim K^{o(1)}\|\mathcal{Q}_{K^{\gamma_1/2}}u_3\|_{X^{0,1/10+}}\lesssim K^{-\gamma_1/5+};\] (notice that $\mathcal{F}\mathcal{Q}_{K^{\gamma_1/2}}u_3(k,\xi)$ is supported where $|k|\lesssim K^{\gamma_0}$ and $|\xi|\gtrsim K^{\gamma_1/2}$, on which $|\xi+ 2\pi Q(k)|\gtrsim K^{\gamma_1/2}$ since $\gamma_1>4\gamma_0+$), and by H\"{o}lder\[\|\mathcal{P}_{K^{\gamma_1/2}}u_3\|_{L_{t,x}^{\infty}([0,1]\times\mathbb{T}^3)}+\|\mathcal{Q}_{K^{\gamma_1/2}}u_3\|_{L_{t,x}^{\infty}([0,1]\times\mathbb{T}^3)}\lesssim \|u_3\|_{L_{t,x}^{\infty}(\mathbb{R}\times\mathbb{T}^3)}\lesssim K^{\gamma_0/2+}.\] 
We also have
 \[\|\Omega_2'\|_{L_{t,x}^{5/2}(\mathbb{R}\times\mathbb{T}^3)}\lesssim K^{-\gamma_1/5+(p-2)\gamma_0/2+}\]
uniformly in $n$, thus the contribution given by $\Omega_2'$ is bounded by\begin{align*}|\mathcal{J}_5|& \lesssim\bigg|\int_{\mathbb{R}\times\mathbb{T}^3} P^*(t)\cdot \overline{z''}\cdot w''\cdot \Omega_2'\bigg|\\ & \lesssim\left\|\frac{z''}{1+|t|^2}\right\|_{L_{t,x}^{10/3-}(\mathbb{R}\times\mathbb{T}^3)} \left\|\frac{w''}{1+|t|^2}\right\|_{L_{t,x}^{10/3}(\mathbb{R}\times\mathbb{T}^3)}\left\| \Omega_2' \right\|_{L_{t,x}^{5/2+}(\mathbb{R}\times\mathbb{T}^3)} \\ & \lesssim K^{-\gamma_1/5+(p-2)\gamma_0/2+}.\end{align*}
 Moreover, the term $\Omega_1'$ can be decomposed as $\Omega^*:=\mathcal{Q}_{K^{\gamma_1}}\Omega_1'$ and $\Omega'':=\mathcal{P}_{K^{\gamma_1}}\Omega_1'$. For the term $\Omega^*$ one has \[\|\Omega^*\|_{L_{t,x}^{5/2+}(\mathbb{R}\times\mathbb{T}^3)}\lesssim K^{-\gamma_1}\|\partial_t \Omega_1'\|_{L_{t,x}^{5/2+}(\mathbb{R}\times\mathbb{T}^3)}\lesssim K^{-\gamma_1}\|\mathcal{P}_{K^{\gamma_1/2}}u_3\|_{L_{t,x}^{10-}(\mathbb{R}\times\mathbb{T}^3)}^{p-2}\|\partial_t \mathcal{P}_{K^{\gamma_1/2}}u_3\|_{L_{t,x}^{q_0}(\mathbb{R}\times\mathbb{T}^3)}\lesssim K^{-\gamma_1/2}\] using the fact that \[\|\mathcal{P}_{K^{\gamma_1/2}}u_3\|_{L_{t,x}^{10-}(\mathbb{R}\times\mathbb{T}^3)}\lesssim \|u_3\|_{L_{t,x}^{10-}}\lesssim 1;\quad \|\partial_t\mathcal{P}_{K^{\gamma_1/2}}u_3\|_{L_{t,x}^{q_0}(\mathbb{R}\times\mathbb{T}^3)}\lesssim K^{\gamma_1/2}\|u_3\|_{L_{t,x}^{q_0}}\lesssim K^{\gamma_1/2},\] so the corresponding contribution is \[|\mathcal{J}_6|\lesssim\bigg|\int_{\mathbb{R}\times\mathbb{T}^3} P^*(t)\cdot \overline{z''}\cdot w''\cdot \Omega^*\bigg|\lesssim\left\|\frac{z''}{1+|t|^2}\right\|_{L_{t,x}^{10/3-}(\mathbb{R}\times\mathbb{T}^3)} \left\|\frac{w''}{1+|t|^2}\right\|_{L_{t,x}^{10/3}(\mathbb{R}\times\mathbb{T}^3)}\left\|\Omega^*\right\|_{L_{t,x}^{5/2+}(\mathbb{R}\times\mathbb{T}^3)} \lesssim K^{-\gamma_1/2+}.\]

Finally, we are left with the term\begin{equation}\label{defj7}\mathcal{J}_7:=\int_{\mathbb{R}\times\mathbb{T}^3}P^*(t)\cdot \overline{z''}\cdot w''\cdot \Omega''=\sum_{k_1=k_2+k_3}\int_{\xi_1=\xi_0+\xi_2+\xi_3}\widehat{P^*}(\xi_0)\overline{\widehat{z''}(k_1,\xi_1)}\widehat{w''}(k_2,\xi_2)\widehat{\Omega''}(k_3,\xi_3).\end{equation} For this term, first we have \[|\mathcal{J}_7|\lesssim\left\|\frac{z''}{1+|t|^2}\right\|_{L_{t,x}^{10/3-}(\mathbb{R}\times\mathbb{T}^3)} \left\|\frac{w''}{1+|t|^2}\right\|_{L_{t,x}^{10/3}(\mathbb{R}\times\mathbb{T}^3)}\left\|\Omega''\right\|_{L_{t,x}^{5/2+}(\mathbb{R}\times\mathbb{T}^3)} \lesssim K^{o(1)},\] since by Strichartz, \[\|\Omega''\|_{L_{t,x}^{5/2+}(\mathbb{R}\times\mathbb{T}^3)}\lesssim \|\Omega_1'\|_{L_{t,x}^{5/2+}(\mathbb{R}\times\mathbb{T}^3)}\lesssim\left\|\mathcal{P}_{K^{\gamma_1/2}}u_3\right\|_{L_{t,x}^{5(p-1)/2+}(\mathbb{R}\times\mathbb{T}^3)}^{p-1}\lesssim\|u_3\|_{L_{t,x}^{5(p-1)/2+}}^{p-1}\lesssim 1.\] Moreover, by the definition of these factors, we know that in the $\xi$-integral (\ref{defj7}), we must have \begin{equation}\label{resnt}\max(|\xi_0|,|\xi_1+ 2\pi Q(k_1)|,|\xi_2+ 2\pi Q(k_2)|,|\xi_3+ 2\pi Q(k_3)|)\ll K^{\gamma_1}.\end{equation} This gives that\[\left|Q(k_1)-Q(k_2)-Q(k_3)\right|\ll K^{\gamma_1},\] which gives that\begin{equation}\label{plane}|Q(k_1,k_3)|\ll K^{\gamma_1},\end{equation} since $k_3=k_1-k_2$ and $|Q(k_3)|\lesssim K^{2\gamma_0}\ll K^{\gamma_1}$, where \[Q(\ell,m):=\sum_{i=1}^3\beta_i\ell_im_i\] denotes the bilinear form corresponding to $Q(k)$. Moreover, since $k_3\neq 0$, we get $k_1\in\mathcal{Y}$, where\[\mathcal{Y}:=\bigcup_{0\neq\ell\in\mathbb{Z}^3,|\ell|\lesssim K^{\gamma_0}}\big\{k\in\mathbb{R}^3:|k|\lesssim K, |Q(k,\ell)|\ll K^{\gamma_1}\big\}\subset\mathbb{R}^3\] is the union of at most $O(K^{3\gamma_0})$ rectangular cuboids of dimensions $K\times K\times O(K^{\gamma_1})$.

This completes the estimate for $\mathcal{R}_3'$. The estimate for $\mathcal{R}_2'$ is done in completely analogous way; in fact, one may first replace the $F_{p-1}(u_1)$ factor by $P_{\leq K/10}F_{p-1}(u_1)$, then argue in exactly the same way as above, the only difference being that we now have $\xi_1+\xi_2=\xi_0+\xi_3$ in the integral (\ref{defj}) due to the presence of $\overline{u_2}$,  \[\max(|\xi_0|,|\xi_1+ 2\pi Q(k_1)|,|\xi_2+2\pi Q(k_2)|,|\xi_3+2\pi Q(k_3)|)\gtrsim K^{2}\] is always true.
\bigskip

\noindent \underline{Step 4: from the estimates on $\mathcal{R}$ to the desired inequality.} Summing up, we get that $\mathcal{R}'$ can be decomposed into two parts, \[\mathcal{R}'=\mathcal{R}''+\widetilde{\mathcal{R}},\] where \[\left\|\chi(t)\mathcal{R}''\right\|_{X^{1,b-1}}\lesssim K^{o(1)}K^{\max(\sigma\gamma_0,-\gamma_1/5+(p-2)\gamma_0/2)},\] and \[\|\chi(t)\widetilde{\mathcal{R}}\|_{X^{1,b-1}}\lesssim K^{o(1)} ,\quad\mathrm{supp}(\mathcal{F}\widetilde{\mathcal{R}})\subset\mathcal{Y}\times\mathbb{R}.\] Using Proposition \ref{linear}, one gets that\begin{equation}\label{approx}\left\|v(t)-e^{it\Delta_{\beta}}v(0)\right\|_{L_{t,x}^{q_0}([0,\varepsilon]\times\mathbb{T}^3)}\lesssim K^{\sigma+}\cdot\max\left(K^{\sigma\gamma_0},K^{-\gamma_1/5+(p-2)\gamma_0/2},K^{3\gamma_0-(p-3)(1-\gamma_1)/6}\right).\end{equation} 
Optimizing the right hand side  leads to choosing $\gamma_0 = \frac{p-3}{3(p+5)}$ and $\gamma_1 = 6 \gamma_0$, which gives
\begin{equation}\label{approx2}\left\|v(t)-e^{it\Delta_{\beta}}v(0)\right\|_{L_{t,x}^{q_0}([0,\varepsilon]\times\mathbb{T}^3)}\lesssim K^{\sigma+\theta_1+}.\end{equation} By time translation, one also gets that \begin{equation}\label{approx3}\left\|v(t)-e^{i(t-m\varepsilon)\Delta_{\beta}}v(m\varepsilon)\right\|_{L_{t,x}^{q_0}([m\varepsilon,(m+1)\varepsilon]\times\mathbb{T}^3)}\lesssim K^{\sigma+\theta_1+}.\end{equation} for each $0\leq m\leq K^{\gamma}$. 
Moreover, using the same arguments as above, one can also prove that\begin{equation}\label{approx4}\sup_{n\in\mathbb{Z}}\left\|e^{i(n\varepsilon+t)\Delta_{\beta}}\left[v((m+1)\varepsilon)-e^{i\varepsilon\Delta_{\beta}}v(m\varepsilon)\right]\right\|_{L_{t,x}^{q_0}([0,\varepsilon]\times\mathbb{T}^3)}\lesssim K^{\sigma+\theta_1+}\end{equation} for each $0\leq m\leq K^{\gamma}$.  In fact, by time translation we may assume $m=0$, so \[v(\varepsilon)-e^{i\varepsilon\Delta_\beta}v(0)=-i\int_0^\varepsilon e^{i(\varepsilon-t')\Delta_{\beta}}\mathcal{R}'(t')\,\mathrm{d}t'.\] Using Proposition \ref{linear}, part (3), and the decomposition $\mathcal{R}'=\widetilde{R}+\mathcal{R}''$ above, we can decompose $v(\varepsilon)-e^{i\varepsilon\Delta_{\beta}}v(0)$ into two terms, one having $H^1$ norm bounded by \[K^{o(1)}K^{\max(\sigma\gamma_0,-\gamma_1/5+(p-2)\gamma_0/2)},\] 
 the other having bounded $H^1$ norm and Fourier transform supported in $\mathcal{Y}$. Then, (\ref{approx4}) follows from Strichartz. Combining (\ref{approx3}) and (\ref{approx4}), one easily gets that\begin{equation}\begin{aligned}\sum_{m=0}^{K^{\gamma}}\left\|v\right\|_{L_{t,x}^{q_0}([m\varepsilon,(m+1)\varepsilon]\times\mathbb{T}^3)}&\lesssim K^{\sigma+\theta_1+\gamma+}+\sum_{m=0}^{K^\gamma}\|e^{i(t-m\varepsilon)\Delta_\beta}v(m\varepsilon)\|_{L_{t,x}^{q_0}([m\varepsilon,(m+1)\varepsilon]\times\mathbb{T}^3)}\\&\lesssim K^{\sigma+\theta_1+\gamma+}+\sum_{m=0}^{K^{\gamma}}\|e^{it\Delta_{\beta}}P_K\mathcal{D}u(0)\|_{L_{t,x}^{q_0}([m\varepsilon,(m+1)\varepsilon]\times\mathbb{T}^3)}\\&+\sum_{m=0}^{K^{\gamma}}\sum_{j=0}^{m-1}\|e^{i(t-(j+1)\varepsilon)\Delta_\beta}v((j+1)\varepsilon)-e^{i(t-j\varepsilon)\Delta_\beta}v(j\varepsilon)\|_{L_{t,x}^{q_0}([m\varepsilon,(m+1)\varepsilon]\times\mathbb{T}^3)}\\
&\lesssim K^{\sigma+\theta_1+2\gamma+}+\sum_{m=0}^{K^{\gamma}}\|e^{it\Delta_{\beta}}P_K\mathcal{D}u(0)\|_{L_{t,x}^{q_0}([m\varepsilon,(m+1)\varepsilon]\times\mathbb{T}^3)}.\end{aligned}\end{equation}
By the long-time Strichartz estimate (part (6) of Proposition \ref{linear}) combined with H\"{o}lder in $m$, one gets that\[\sum_{m=0}^{K^{\gamma}}\left\|v\right\|_{L_{t,x}^{q_0}([m\varepsilon,(m+1)\varepsilon]\times\mathbb{T}^3)}\lesssim K^{\sigma+\theta_1+2\gamma+}+K^{\sigma+\gamma-\gamma/q_0+}.\]
This completes the proof.
\end{proof}
\begin{proof}[Proof of Theorem \ref{main} in the irrational case] By Proposition \ref{long}, choosing $\gamma=q_0|\theta_1|/(q_0+1)$, one has that, as long as $T>K^{|\theta_1|+}$ and $E[\mathcal{D}u(t)]\lesssim 1$ for all $t\in[0,T]$, \[\sum_{m=0}^{\varepsilon^{-1}T}\|P_K\mathcal{D}u\|_{L_{t,x}^{q_0}([m\varepsilon,(m+1)\varepsilon]\times\mathbb{T}^3)}\lesssim TK^{\sigma+}\cdot K^{\frac{\theta_1}{q_0+1}} .\]

Given initial data, choose $N$ such that $N\sim A$ and $E[\mathcal{D}u(0)]\leq 10E$. By Propositions \ref{incre}, as long as \[\sup_{0\leq t\leq T}E[\mathcal{D}u(t)]\leq 20E\] for all $t\in[0,T]$, one has that

\begin{align*}\sup_{0\leq t\leq T}E[\mathcal{D}u(t)] - 10E & \lesssim N^{\max(p-5,-1)}T+N^{o(1)}\sum_K K^{o(1)}\min(1,N^{-1}K)\cdot TK^{\sigma+}\cdot K^{\frac{\theta_1}{q_0+1}}\\
& \lesssim N^{\max(p-5,-1)}T+N^{\sigma+\frac{\theta_1}{q_0+1}+}T.\end{align*}

 By a bootstrap argument, this gives that $E[\mathcal{D}u(t)]\lesssim 1$ and $\|u(t)\|_{H^2}\lesssim A$, up to time \[T=N^{\frac{5-p}{2}+\frac{\theta(p)}{2}},\] where $\theta(p)<2|\theta_1|/(q_0+1)$. By the same argument as in the general case, this implies (\ref{growth2}).
\end{proof}

\end{document}